\documentclass[a4paper]{article}

\usepackage{amsmath,amssymb,amsthm}
\usepackage{color}
\usepackage{url}
\usepackage{enumerate,enumitem}

\usepackage{graphics}
\usepackage{caption}
\usepackage{xcolor}
\usepackage{pgf}

\allowdisplaybreaks

\newcommand{\half}{\frac{1}{2}}
\newcommand{\ii}{\mathcal{I}}
\newcommand{\bby}{\mbox{{\boldmath $\mathcal{Y}$}}}
\newcommand{\bbw}{\mbox{{\boldmath $\mathcal{W}$}}}
\newcommand{\by}{\mathbf{Y}}
\newcommand{\bw}{\mathbf{W}}
\newcommand{\rr}{\mathbb{R}}
\newcommand{\one}{\mathbf{1}}
\newcommand{\eps}{\varepsilon}

\newcommand\raisepunct[1]{\,\mathpunct{\raisebox{0.5ex}{#1}}}

\renewenvironment{proof}[1][\proofname]{\par \normalfont \trivlist
\item[\hskip\labelsep\itshape #1]\ignorespaces
}{%
\hspace*{\fill}$\Box$ \endtrivlist }
\renewcommand{\proofname}{{\bf Proof\,}}

\makeatletter
\def\newrmtheorem#1{\@ifnextchar[{\@rmothm{#1}}{\@rmnthm{#1}}}

\def\@rmnthm#1#2{%
\@ifnextchar[{\@rmxnthm{#1}{#2}}{\@rmynthm{#1}{#2}}}

\def\@rmxnthm#1#2[#3]{\expandafter\@ifdefinable\csname #1\endcsname
{\@definecounter{#1}\@addtoreset{#1}{#3}%
\expandafter\xdef\csname the#1\endcsname{\expandafter\noexpand
  \csname the#3\endcsname \@rmthmcountersep \@rmthmcounter{#1}}%
\global\@namedef{#1}{\@rmthm{#1}{#2}}\global\@namedef{end#1}{\@endrmtheorem}}}

\def\@rmynthm#1#2{\expandafter\@ifdefinable\csname #1\endcsname
{\@definecounter{#1}%
\expandafter\xdef\csname the#1\endcsname{\@rmthmcounter{#1}}%
\global\@namedef{#1}{\@rmthm{#1}{#2}}\global\@namedef{end#1}{\@endrmtheorem}}}

\def\@rmothm#1[#2]#3{\expandafter\@ifdefinable\csname #1\endcsname
  {\global\@namedef{the#1}{\@nameuse{the#2}}%
\global\@namedef{#1}{\@rmthm{#2}{#3}}%
\global\@namedef{end#1}{\@endrmtheorem}}}

\def\@rmthm#1#2{\refstepcounter
    {#1}\@ifnextchar[{\@rmythm{#1}{#2}}{\@rmxthm{#1}{#2}}}

\def\@rmxthm#1#2{\@beginrmtheorem{#2}{\csname the#1\endcsname}\ignorespaces}
\def\@rmythm#1#2[#3]{\@opargbeginrmtheorem{#2}{\csname
       the#1\endcsname}{#3}\ignorespaces}

\def\@rmthmcounter#1{\noexpand\arabic{#1}}
\def\@rmthmcountersep{}
\def\@beginrmtheorem#1#2{\rm \trivlist
      \item[\hskip \labelsep{\bf #1\ #2\thmrmcounterend}]}
\def\@opargbeginrmtheorem#1#2#3{\rm \trivlist
      \item[\hskip \labelsep{\bf #1\ #2\ (#3)\thmrmcounterend}]}
\def\@endrmtheorem{\endtrivlist}
\def\thmrmcounterend{\hskip 0em\relax}
\def\newrmwntheorem#1#2{\expandafter\@ifdefinable\csname #1\endcsname%
\global\@namedef{#1}{\@rmwnthm{#1}{#2}}%
\global\@namedef{end#1}{\@endrmwntheorem}}

\def\newsltheorem#1{\@ifnextchar[{\@slothm{#1}}{\@slnthm{#1}}}

\def\@slnthm#1#2{%
\@ifnextchar[{\@slxnthm{#1}{#2}}{\@slynthm{#1}{#2}}}

\def\@slxnthm#1#2[#3]{\expandafter\@ifdefinable\csname #1\endcsname
{\@definecounter{#1}\@addtoreset{#1}{#3}%
\expandafter\xdef\csname the#1\endcsname{\expandafter\noexpand
  \csname the#3\endcsname \@slthmcountersep \@slthmcounter{#1}}%
\global\@namedef{#1}{\@slthm{#1}{#2}}\global\@namedef{end#1}{\@endsltheorem}}}

\def\@slynthm#1#2{\expandafter\@ifdefinable\csname #1\endcsname
{\@definecounter{#1}%
\expandafter\xdef\csname the#1\endcsname{\@slthmcounter{#1}}%
\global\@namedef{#1}{\@slthm{#1}{#2}}\global\@namedef{end#1}{\@endsltheorem}}}

\def\@slothm#1[#2]#3{\expandafter\@ifdefinable\csname #1\endcsname
  {\global\@namedef{the#1}{\@nameuse{the#2}}%
\global\@namedef{#1}{\@slthm{#2}{#3}}%
\global\@namedef{end#1}{\@endsltheorem}}}

\def\@slthm#1#2{\refstepcounter
    {#1}\@ifnextchar[{\@slythm{#1}{#2}}{\@slxthm{#1}{#2}}}

\def\@slxthm#1#2{\@beginsltheorem{#2}{\csname the#1\endcsname}\ignorespaces}
\def\@slythm#1#2[#3]{\@opargbeginsltheorem{#2}{\csname
       the#1\endcsname}{#3}\ignorespaces}

\def\@slthmcounter#1{.\noexpand\arabic{#1}}
\def\@slthmcountersep{}
\def\@beginsltheorem#1#2{\sl \trivlist
      \item[\hskip \labelsep{\bf #1\ #2\thmslcounterend}]}
\def\@opargbeginsltheorem#1#2#3{\sl \trivlist
      \item[\hskip \labelsep{\bf #1\ #2\ (#3)\thmslcounterend}]}
\def\@endsltheorem{\endtrivlist}
\def\thmslcounterend{\hskip 0em\relax}
\def\newslwntheorem#1#2{\expandafter\@ifdefinable\csname #1\endcsname%
\global\@namedef{#1}{\@slwnthm{#1}{#2}}%
\global\@namedef{end#1}{\@endslwntheorem}}

\makeatother

\newsltheorem{theorem}{Theorem}[section]
\newsltheorem{proposition}[theorem]{Proposition}
\newsltheorem{lemma}[theorem]{Lemma}
\newsltheorem{corollary}[theorem]{Corollary}

\newrmtheorem{remark}[theorem]{Remark}
\newrmtheorem{problem}[theorem]{Problem}
\newrmtheorem{example}[theorem]{Example}

\theoremstyle{definition}

\newtheorem{algorithm}[theorem]{Algorithm}

\begin{document}

\providecommand{\keywords}[1]
{
  \small	
  \textsl{Keywords:} #1
}

\providecommand{\ams}[1]
{
  \small	
  \textsl{AMS subject classification:} #1
}

\title{The inverse problem of positive autoconvolution}%
\author{Lorenzo Finesso\thanks{Lorenzo Finesso is with the Institute of Electronics, Information Engineering and Telecommunications,
National Research Council, CNR-IEIIT, Padova; email: {\tt finesso@ieiit.cnr.it}}
 \and
        Peter Spreij\thanks{Peter Spreij is with the Korteweg-de Vries Institute for Mathematics,
Universiteit van Amsterdam and with the  Institute for Mathematics, Astrophysics and Particle Physics, Radboud University, Nijmegen; e-mail: {\tt spreij@uva.nl}}
\thanks{This work was partially funded under the STM-2020 CNR Fund, project
number 123456. The material in this paper was not presented at any conference.}
}

\maketitle
\begin{abstract}
We pose the problem of approximating optimally a given nonnegative signal with the scalar autoconvolution of a nonnegative signal. The I-divergence is chosen as the optimality criterion being well suited to incorporate nonnegativity constraints. After proving the existence of an optimal approximation we derive an iterative  descent algorithm of the alternating minimization type to find a minimizer. The algorithm is based on the lifting technique developed by Csisz\'ar and Tusn\'adi and exploits the optimality properties of the related minimization problems in the larger space. We study the asymptotic behavior of the iterative algorithm and prove, among other results, that its limit points are Kuhn-Tucker points of the original minimization problem. Numerical experiments confirm the asymptotic results and exhibit the fast convergence of the proposed algorithm.


\ams{93B30, 94A17}

\end{abstract}

\section{Introduction}

Inverse problems in system modeling and identification have a long tradition and have been the subject of a vast technical literature in applied mathematics, engineering, and specialized applied fields. The classic book~\cite{tikhonov} surveys the early contributions to the field.
The focus of this paper is on the subclass of inverse problems for which the models are of \emph{autoconvolution} type.  In linear time invariant systems, inputs are transformed into outputs by convolution with a kernel representing the system's impulse response. Autoconvolution systems produce the output by convolution of the input signal with itself.

A lot of work has been dedicated to the inverse problem of autoconvolution for functions on the real line, emphasizing the functional analytic aspects and motivating its interest in a variety of applications in physics and engineering.
Most of the contributions analyse special cases, where exact solutions to the inverse problem exist, and propose different theoretical approaches for their construction.
The paper \cite{dose1979deconvolution}, for example, focuses on inversion of autoconvolution integrals using spline functions, whereas
\cite{dose1979deconvolutionII} performs inversion by polygonal approximation.
In \cite{martinez1979global} inversion is studied based on the application of the FFT algorithm and digital signal processing concepts.
Special cases arise when dealing with autoconvolution of probability density functions, as in \cite{hofmann2006determination}.
In \cite{douglas2014autoconvolution} the autoconvolution has been introduded for continuous time processes as an alternative to autocorrelation. Ill-posedness issues and Tykhonov regularization are omnipresent, see for instance \cite{burger2015deficit}, \cite{gorenflo1994autoconvolution}.

The main differences between the cited literature and this paper are that we consider approximation problems, rather than looking for exact solutions which exist only exceptionally, and that our (time) domain is discrete rather than the real line. Moreover, the nonnegativity constraint, that we impose on all signals, is a crucial feature of the present work.
Some earlier work shares, at least in part, our point of view, e.g. the papers~\cite{choilanterman2005}, \cite{choilantermanraich2006} dealing with image processing and 2D systems, contain an algorithm of the same type as ours and an analysis of its behavior. In~\cite{schulz2005signal} an algorithm similar to ours is set up to solve a problem of signal recovery using auto and cross correlations instead of autoconvolutions.

The purpose of this paper is threefold. First we pose the problem of a time-domain approximation of a nonnegative input/output systems by finite autoconvolutions when the output observations are available. Following the choice made in other optimization problems for nonnegative system, we opt for the I-divergence, which as argued in~\cite{cs1991} (see also \cite{snyderetal1992}), is the natural choice for approximation problems under nonnegativity constraints. We provide a result on the existence of the minimizer of the approximation criterion.
Then we propose an iterative algorithm to find the best approximation, and finally we study the asymptotical behavior of the algorithm.

We employ techniques that have already been used in~\cite{fs2006} to analyse a nonnegative matrix factorization problem and the approach is similar to the one in \cite{finessospreij2015ieeeit}, \cite{finessospreij2019automatica}, but differs from the latter references as they treat linear convolutional problems, whereas the autoconvolution is inherently nonlinear.
The algorithm that we propose is of the alternating minimization type, and the optimality conditions (the Pythagorean relations) are satisfied at each step.

The inherent non convexity, and nonlinearity of the problem make the analysis of the asymptotic behavior challenging. The main result in this respect is contained in Proposition \ref{prop:kt} which states that all limit points of the algorithm satisfy the Kuhn Tucker optimality conditions. This should be compared with other known results on the convergence of alternating minimization algorithms. In some cases it is possible to show convergence to a (unique) limit, which is also the minimizer of the criterion. This happens, in particular, when dealing with a convex criterion. Contributions in this direction are e.g.~\cite{cover1984}, \cite{snyderetal1992}, and \cite{finessospreij2015ieeeit}, \cite{finessospreij2019automatica}, \cite{Vardietal1985}. On the other hand for non convex, nonlinear problems, to the best of our knowledge, there are no asymptotic results comparable with the present Proposition \ref{prop:kt}.

It must be remarked that the nonparametric approach to the inverse problem, that we follow in this paper, is different from the one followed in identification or realization of nonnegative and linear systems, see~\cite{benvenutifarina2004} for a survey, and for instance~\cite{andersondeistler1996}, \cite{farina1995}, \cite{gurvits2007}, \cite{nagy2005}, \cite{nagy2007}, \cite{shu2008}.

A brief summary of the paper follows. In Section~\ref{section:problem} we state the problem and show the existence of a solution and give some of its properties. In Section~\ref{section:lift} the original problem is lifted into a higher dimensional setting, thus making it amenable to alternating minimization. The optimality properties (Pythagoras rules) of the ensuing partial minimization problems are discussed here. After that we derive in Section~\ref{section:algo} the iterative minimization algorithm combining the solutions of the partial minimizations, and analyse the convergence properties. In particular we show that limit points of the algorithm are Kuhn-Tucker points of  the original optimization problem. In the concluding Section~\ref{section:numerics} we present numerical experiments that show the quick  convergence of the algorithm and corroborate the theoretical results on its asymptotic behaviour.

\section{Problem statement and initial results}\label{section:problem}

In the paper we consider real valued signals $x: \mathbb Z \rightarrow \mathbb R$, mapping $i \mapsto x_i$, that vanish for $i<0$, i.e., $x_i=0$ for $i<0$. The \emph{support} of $x$ is the discrete time interval $[0, n]$, where $n=\inf\{\,k:\,\, x_i=0,\,\,\, \text{for $i> k$}\,\}$, if the infimum is finite (and then a minimum), and $[0, \infty)$ otherwise. The autoconvolution of $x$ is the signal $x* x$, vanishing for $i<0$, and satisfying,
\begin{equation}\label{eq:xconv}
(x* x)_i =\sum_{j=-\infty}^\infty x_{i-j}x_j =
                 \sum_{j=0}^i x_{i-j}x_j\,,  \qquad i \ge 0.
\end{equation}
Notice that if the support of $x$ is finite $[0, n]$, the support of $x* x$ is $[0,2n]$. In this case, when computing $(x * x)_i$ for $i>n$, the summation in Equation~\eqref{eq:xconv} has non zero addends only in the range $i-n\le j \le n$, as $x_{i-j}=0$ and $x_j=0$ for $i-j>n$ and $j>n$ respectively. If the signal $x$ is nonnegative, i.e.\ $x_i \ge0$ for all $i\in \mathbb Z$,  the autoconvolution (\ref{eq:xconv}) is too.
Given a finite \emph{nonnegative} data sequence
$$
y=(y_0,\dots , y_n),
$$
the problem is finding a \emph{nonnegative} signal $x$ whose autoconvolution $x* x$ best approximates $y$. Since the signals involved are nonnegative, the approximation criterion is chosen to be the I-divergence, see~\cite{c1975,cs1991}. The I-divergence between two nonnegative vectors $u$ and $v$ of equal length is
\[
\ii(u,v)=\sum_i u_i\log\frac{u_i}{v_i}-u_i+v_i\,,
\]
if $u_i=0$ whenever $v_i=0$, and $\ii(u,v)=\infty$ if there exist an index $i$ with $u_i>0$ and $v_i=0$. It is known that $\ii(u,v)\geq 0$, with equality iff $u=v$.

Depending on the constraints imposed on the support of $x$ the basic problem splits into two different cases. The first case involves a full length signal $x=(x_0,\dots, x_n)$ and produces the approximation problem specified below, where we write $x* x\in\rr^{n+1}$ for the restriction to $[0, n]$ of the convolution $x* x$ defined in~\eqref{eq:xconv}.
\begin{problem}\label{problem}
Given $y\in\rr^{n+1}_+$ minimize, over $x\in\rr^{n+1}_+=[0,\infty)^{n+1}$,
\begin{equation} \label{cost-full}
\ii=\ii(x):=\ii(y||x* x)=\sum_{i=0}^n\Big(y_i\log\frac{y_i}{(x* x)_i}-y_i+(x* x)_i\Big)\,.
\end{equation}
\end{problem}
In alternative, recalling that, in the finite case, the support of $x* x$ is twice the support of $x$, one can consider, when $n=2m$, the problem of approximating the given data $y=(y_0,\dots, y_{2m})$ with the autoconvolution $x* x$ of a signal of half length,  $x=(x_0,\dots, x_m)$. This leads to the following approximation problem.
\begin{problem}\label{problemhalf}
Given $y\in\rr^{2m+1}_+$  minimize, over $x\in\rr^{m+1}_+=[0,\infty)^{m+1}$,
\begin{equation} \label{cost-half}
\ii=\ii(x):=\ii(y||x* x)=\sum_{i=0}^{2m}\Big(y_i\log\frac{y_i}{(x* x)_i}-y_i+(x* x)_i\Big)\,.
\end{equation}
\end{problem}
Notice that if the given data are $y=(y_0,\dots,y_n)$ with $n$ odd, i.e.\ $n=2m-1$ for some integer $m\geq 1$, one can still pose Problem~\ref{problemhalf} with $x\in\rr^{m+1}_+$, simply introducing the fictitious data point $y_{2m}=0$. Hence in Problem~\ref{problemhalf}, without loss of generality, the number of data points will always be assumed odd, that is we assume $n$ even, $n=2m$.

Note that Problem~\ref{problem}, under the constraint that the support of $x$ is $[0, m]$, where $m=\lfloor \frac{n+1}{2}\rfloor$, reduces to Problem \ref{problemhalf}. Although the latter is a constrained version of the former problem and the approaches to their solutions are similar, the analysis and the results are very different. In this paper we concentrate on Problem \ref{problemhalf} which is easier to analyse and produces an algorithm with a much simpler structure. Problem~\ref{problem} will be investigated in a future publication.

The objective function \eqref{cost-half} is non convex and nonlinear in $x$, the existence of a minimizer is therefore not immediately clear. Our first result settles in the affirmative the question of the existence. The issue of uniqueness remains open, but we have evidence of the existence of multiple local minima of $\ii(x)$. See Section~\ref{section:numerics} for numerical examples.
\begin{proposition}\label{proposition:exist}
Problem~\ref{problemhalf} admits a  solution.
\end{proposition}

\begin{proof}
Let $x=x^0$ be an arbitrary vector in $\rr^{m+1}_+$. Performing one step of Algorithm~\ref{algorithm:half}, introduced below, yields the iterate $x^1$ satisfying  $\ii(x^1)\leq \ii(x)$ and $(\sum_{i=0}^mx^1_i)^2=\sum_{i=0}^{2m}y_i$, by virtue of Proposition~\ref{proposition:properties}. The search for a minimizer can hence be limited to the compact subset $K_0\subset \mathbb R^{m+1}_+$ of the $x$'s satisfying $(\sum_{i=0}^mx^1_i)^2=\sum_{i=0}^{2m}y_i$.
Noting that $\ii(x)=\sum_{i:y_i>0}(y_i\log\frac{y_i}{(x* x)_i}-y_i)+\sum_i(x* x)_i$, we can restrict attention even further to those $x$'s for which $(x* x)_i\geq \eps$ for all $i$ such that $y_i>0$, by choosing $\eps$ sufficiently small and positive.
This implies that we restrict the finding of the minimizers to an even smaller compact set $K_1$ on which $\ii$ is continuous.
This proves the existence of a minimizer.
\end{proof}

\noindent
A basic ingredient for the minimization of the cost~(\ref{cost-half}) is its gradient which is computed below. As a preliminary step note that
\[
\frac{\partial}{\partial x_j}(x* x)_i = \left\{\begin{array}{ll}
                 2x_{i-j}, &\text{for}\,\,\, 0\le j\le m,\,\,\,j\le i\le j+m\\
                 0, & \mbox{otherwise}\,,
                \end{array}
         \right.
\]
therefore
\begin{align}
\nabla_j\ii(x) & :=\frac{\partial \ii(x)}{\partial x_j} = \frac{\partial}{\partial x_j} \Big( \sum_{i=0}^{2m} -y_i\log (x* x)_i + (x* x)_i\Big) \nonumber\\
&= 2\sum_{i=j}^{j+m} \Big(-x_{i-j} \frac{y_i}{(x* x)_i} + x_{i-j} \Big)
 = 2\sum_{\ell=0}^m \Big(- x_\ell \frac{y_{\ell+j}}{(x* x)_{\ell+j}} + x_\ell \Big)\,.\label{eq:gradj}
\end{align}
Equations \eqref{eq:gradj} are highly nonlinear in $x$ and solving the first order optimality conditions $\nabla \ii(x)=0$, where $\nabla$ denotes the gradient vector, to find the stationary points of \eqref{cost-half}, will not result in analytic solutions except in trivial cases. This observation calls for a numerical approach to the optimization, which we will present in Section~\ref{section:algo}.

\bigskip\noindent
The following result shows a useful property of the minimizers of $\ii(x)$.
\begin{proposition}\label{prop:sum}
For any $x\in\mathbb R^{m+1}$ it holds that
\begin{equation}\label{eq:propconv}
\sum_{i=0}^{2m}(x* x)_i= \Big(\sum_{i=0}^{2m} x_i\Big)^2.
\end{equation}
Moreover, if $x^\star\in\mathbb R^{m+1}_+$ is a minimizer of Problem \ref{problemhalf},
\begin{equation}\label{eq:propminim}
\sum_{i=0}^{2m}(x^\star* x^\star)_i= \Big(\sum_{i=0}^{2m} x^\star_i\Big)^2= \sum_{i=0} ^{2m}y_i\,.
\end{equation}
\end{proposition}
\begin{proof}
The identity (\ref{eq:propconv}) is a general property, indeed for any $x$,
\begin{align*}
\sum_{i=0}^{2m} (x* x)_i & = \sum_{i=0}^{2m} \sum_{j=0}^i x_{i-j} x_j  = \sum_{j=0}^{2m} \sum_{i=j}^{2m} x_{i-j} x_j\\
& = \sum_{j=0}^{2m} x_j\sum_{i=j}^{2m} x_{i-j} = \Big(\sum_{j=0}^{m} x_j\Big)^2.
\end{align*}
\noindent
To prove identity (\ref{eq:propminim}), let $x^\star$ be a minimizer of $\ii(x)$ and define $f(\alpha)=\ii(\alpha x^\star)$, for $\alpha>0$. It follows that $f'(1)=0$. A direct computation of $f'(\alpha)$ gives $f'(\alpha)=-\frac{2}{\alpha}\sum_{i=0} ^{2m}y_i+2\alpha\sum_{i=0}(x^\star * x^\star)_i$, hence $f'(1)=0$ yields the wanted identity.
\end{proof}

\begin{remark} If $y$ is strictly positive the I-divergence in \eqref{cost-half} vanishes if and only if $y_i=(x* x)_i$ for all $i\in [0,2m]$. That is the (special) case where an exact solution to the deautoconvolution problem exists. Notice that this is a non generic case as the $2m+1$ equations $y_i=(x* x)_i$ in the $m+1$ variables $x$ specify an (at most) $(m+1)$-dimensional submanifold in the data space $\mathbb R^{2m+1}_+$. See the  example below for an illustration.
\end{remark}

\begin{example} \label{expl:trivial-half-case}
For $m=1$, let $y=(y_0, y_1, y_2)$ be the given data. Setting the gradient $\nabla \ii(x)=0$ one gets the unique minimizer $x^\star=(x_0^\star,x_1^\star)$ as
\begin{align*}
x_0^\star=\frac{2y_0+y_1}{2\sqrt{y_0+y_1+y_2}}\raisepunct{,} \quad
x_1^\star=\frac{2y_2+y_1}{2\sqrt{y_0+y_1+y_2}}\raisepunct{.}
\end{align*}
One easily verifies that $x^\star$ satisfies property~(\ref{eq:propminim}). Note that this solution, in general, does not give a perfect match; e.g.\ it should hold that $(x^\star * x^\star)_0=(x_0^\star)^2= y_0$. In fact, a necessary and sufficient condition on $y$ that insures the existence of the exact solution, i.e.\   $\ii(y||x^\star* x^\star)=0$, is $y_1^2=4y_0y_2$.
\end{example}

\begin{remark}
Problem~\ref{problemhalf} has an interesting probabilistic interpretation when $\sum_{i=0}^{2m}y_i=1$. The $y_i$ can then be considered as the distribution of a random variable $Y$ taking on $2m+1$ different values. The problem is then to find the optimal distribution of independent and identically distributed random variable $X_1$  and $X_2$ (assuming $m+1$ values) such that $Y=X_1+X_2$. Note that Proposition~\ref{prop:sum} guarantees that the optimal vector $x^\star$ indeed has the interpretation of a distribution.
In Example~\ref{expl:trivial-half-case}, with $y_0+y_1+y_2=1$, the optimal distribution is then $(x_0,x_1)=(y_0+\half y_1,y_2+\half y_1)$. As now one has $y_1=1-y_0-y_2$, it follows that $(x_0,x_1)=\half(y_0-y_2+1,y_2-y_0+1)$ and the perfect match condition reduces to $\sqrt{y_0}+\sqrt{y_2}=1$, in which case of course $X_1$ and $X_2$ can be thought of having a Bernoulli distribution and $Y$ a binomial distribution.
\end{remark}

\section{Lifting and partial minimizations}\label{section:lift}

In this section Problem~\ref{problemhalf} is recast as a double minimization problem by lifting it into a larger space.
The ambient spaces for the lifted problem are the subsets $\bby$ and $\bbw$, defined below, of the set of matrices $\rr^{(2m+1)\times (m+1)}_+$,
$$
\bby := \big\{\,\by  : \,\, \by_{ij}=0,\quad \text{for}\,\, 0\le i < j \,\, \text{and}\,\, i> j+m, \,\, \text{and}\,\, \textstyle{\sum_j}\by_{ij} = y_i\, \big\}\,,
$$
with $y=(y_0,\dots, y_{2m})\in\mathbb R^{2m+1}_+$ the given data vector, and
$$
\bbw := \big\{\bw  : \,\,\bw_{ij} = x_{i-j}x_j,\,\text{if}\,\, 0\le j\le m,\, j\le i \le j+m;\,\, \bw_{ij}=0\,\,\text{otherwise} \big\}\,.
$$
The structure of the matrices in $\bby$ and $\bbw$ is shown below for $m=3$,

\small
$$
\by=\begin{bmatrix}
\by_{00} & 0 & 0 & 0\\
\by_{10} & \by_{11}& 0 & 0\\
\by_{20} & \by_{21}& \by_{22} & 0\\
\by_{30} & \by_{31}& \by_{32} & \by_{33}\\
0 & \by_{41}& \by_{42} & \by_{43}\\
0 & 0& \by_{52} & \by_{53}\\
0 & 0& 0 & \by_{63}
\end{bmatrix},\qquad
\bw=\begin{bmatrix}
x_0x_0 & 0 & 0 & 0\\
x_1x_0 & x_0x_1& 0 & 0\\
x_2x_0 & x_1x_1& x_0x_2 & 0\\
x_3x_0 & x_2x_1& x_1x_2 & x_0x_3\\
0 & x_3x_1& x_2x_2 & x_1x_3\\
0 & 0& x_3x_2 & x_2x_3\\
0 & 0& 0 & x_3x_3
\end{bmatrix}\,.
$$
\normalsize

\bigskip\noindent
The interpretation is as follows. The matrices $\by\in\bby$ and $\bw\in\bbw$ have common support on the the diagonal and first $m$ subdiagonals of $\rr^{(2m+1)\times (m+1)}_+$. The row marginal (i.e.\ the column vector of row sums) of any $\by\in \bby$ coincides with the given data vector $y$. The elements of the $\bw$ matrices factorize, equivalently their row marginal is the autoconvolution of the column marginal rescaled by $1/\sum_i x_i$.

\bigskip\noindent
We introduce two partial minimization problems over the subsets $\bby$ and $\bbw$. Recall that the I-divergence between two nonnegative matrices of the same sizes $M, N\in\rr^{p\times q}_+$ is defined as
\[
\ii(M||N) := \sum_{i,j} \Big( M_{ij} \log \frac{M_{ij}}{N_{ij}} - M_{ij} + N_{ij}\Big)\,.
\]

\begin{problem}\label{problemyhalf}\,
Given $\bw\in\bbw$, minimize $\ii(\by||\bw)$ over $\by\in\bby$.
\end{problem}

\begin{problem}\label{problemwhalf}\,
Given $\by\in\bby$, minimize $\ii(\by||\bw)$ over $\bw\in\bbw$.
\end{problem}
The solutions to both problems can be given in closed form.

\begin{lemma}\label{lemmayhalf}
Problem~\ref{problemyhalf} has the explicit minimizer $\by^\star=\by^\star(\bw)$ given by
\begin{align}
\by^\star_{ij} & =
\frac{\bw_{ij}}{\sum_j\bw_{ij}}\,y_i
 = \left\{
\begin{array}{ll}\dfrac{x_{i-j}x_j}{(x* x)_i}\,y_i & \mbox{ if }\, 0\le j\leq i \leq j+ m,\\ \\
0 & \mbox{ otherwise}.
\end{array}
\right. \label{eq:ystarhalf}
\end{align}
Moreover the \emph{Pythagorean identity}
\begin{equation}\label{eq:pyth-1}
\ii(\by||\bw)=\ii(\by||\by^\star)+\ii(\by^\star||\bw)\,,
\end{equation}
holds for any $\by\in\bby$, and
\begin{equation}\label{eq:fallback}
\ii(\by^\star||\bw)=\ii(y||x* x)\,.
\end{equation}
\end{lemma}

\begin{proof}
Proceed by direct computation. The Lagrangian function is
\[
L=\sum_{ij}\big(\by_{ij}\log \by_{ij}-\by_{ij}\log \bw_{ij}-\by_{ij}+\bw_{ij}\big)-\sum_i\lambda_i\big(\sum_j\by_{ij}-y_i\big)\,,
\]
therefore
\[
\frac{\partial L}{\partial \by_{ij}}=\log\by_{ij}-\log\bw_{ij}-\lambda_i=0\,,
\]
yields $\by_{ij}=\bw_{ij}e^{\lambda_i}$ and imposing the marginal constraint $\sum_j \by_{ij}=y_i$ one gets the asserted minimizer~\eqref{eq:ystarhalf}. Next, introducing the notation $\bw_{i\,\cdot}=\sum_j\bw_{ij}$ and $\by_{i\,\cdot}=\sum_j\by_{ij}$, substitution into the RHS of \eqref{eq:pyth-1} gives
\begin{align*}
\ii(\by &||\by^\star) + \ii(\by^\star||\bw) \\
&= \sum_{ij}\Big(\by_{ij}\log \frac{\by_{ij}}{\by^\star_{ij}}-\by_{ij}+\by^\star_{ij}\Big)
+ \Big(\by^\star_{ij}\log \frac{\by^\star_{ij}}{\bw_{ij}}-\by^\star_{ij}+\bw_{ij}\Big) \\
& = \sum_{ij}\Big(\by_{ij}\log \frac{\by_{ij}}{\bw_{ij}}-\by_{ij}\log \frac{y_i}{\bw_{i\,\cdot}}-\by_{ij}\Big)
 + \Big(\frac{y_i}{\bw_{i\,\cdot}}\bw_{ij}\log\frac{y_i}{\bw_{i\,\cdot}} +\bw_{ij}\Big) \\
& = \ii(\by||\bw)\,,
\end{align*}
thus  proving~\eqref{eq:pyth-1}. As a byproduct of the Pythagorean identity one gets that $\by^\star$ is indeed a minimizer for Problem \ref{problemyhalf}. Finally, using $\bw_{ij}=x_{i-j}x_j$, and $\bw_{i\,\cdot}=(x* x)_i$ one finds that the optimal value of Problem~\ref{problemyhalf} coincides with~\eqref{eq:fallback}. Indeed,
\begin{align*}
\ii(\by^\star||\bw)
& = \sum_{ij}\Big(\bw_{ij}\frac{y_i}{\bw_{i\,\cdot}}\log\frac{y_i}{\bw_{i\,\cdot}}-\bw_{ij}\frac{y_i}{\bw_{i\,\cdot}}+\bw_{ij}\Big) \\
& = \sum_{i} \Big(y_i\log\frac{y_i}{\bw_{i\,\cdot}}-y_i+\bw_{i\,\cdot}\Big) \\
& = \sum_{i}\Big(y_i\log\frac{y_i}{(x* x)_i}-y_i+(x* x)_i\Big) \\
& = \ii(y||x* x)\,.
\end{align*}
\end{proof}

\begin{remark} \label{rem:sym-y} Note that the minimizer $\by^\star$ in~\eqref{eq:ystarhalf} exhibits always the following symmetry
\begin{equation} \label{eq:sym-y}
\by^\star_{j+\ell, \ell} =  \by^\star_{j+\ell, j}\,,\quad \text{for all}\,\,\, \ell, j=0,\dots, m\,,
\end{equation}
i.e., for all $j=0,\dots, m$, the $j$-th subdiagonal of $\by^\star$ and the $(\by_{j,j}, \dots, \by_{j+m,j})^\top$ subvector of its $j$-th column coincide.
\end{remark}

\begin{lemma}\label{lemmawhalf}
Problem~\ref{problemwhalf} has explicit minimizer $\bw^\star=\bw^\star(\by)$ corresponding to $x^\star_j$ as follows,
\begin{equation}\label{eq:xstarhalf}
x_j^\star=\frac{\widehat\by_j}{2\sqrt{\sum_{i=0}^{2m}y_i}}\,\raisepunct{,} \qquad j=0,\dots, m\,,
\end{equation}
where
\begin{equation}\label{eq:yhatl}
\widehat\by_j:=\sum_{i=0}^m\by_{i+j,i} + \sum_{i=j}^{j+m}\by_{ij}\,, \qquad j=0,\dots, m\,.
\end{equation}
Moreover the \emph{Pythagorean identity}
\begin{equation}\label{eq:pyth-2}
\ii(\by||\bw)=\ii(\by||\bw^\star)+\ii(\bw^\star||\bw)\,,
\end{equation}
holds for any $\bw\in\bbw$.
\end{lemma}

\begin{proof}
Minimizing the I-divergence
\[
\ii(\by||\bw)=\sum_{j=0}^m\sum_{i=j}^{j+m} \,\Big(\,\by_{ij}\log\frac{\by_{ij}}{\bw_{ij}}- \by_{ij}+ \bw_{ij}\,\Big)\,,
\]
with respect to $\bw\in \bbw$, is equivalent, since $\bw_{ij}=x_{i-j}x_j$, to minimizing
\begin{align}\label{eq:fwhalf}
F(x) &:= \sum_{j=0}^m\sum_{i=j}^{j+m}\Big(-\by_{ij}\log (x_{i-j}x_j)  + x_{i-j}x_j\Big) \nonumber\\
& = \sum_{j=0}^m\sum_{i=j}^{j+m} \Big(-\by_{ij}\log x_{i-j} - \by_{ij}\log x_j + x_{i-j}x_j \Big)\,.
\end{align}
Applying to the first and third double sums in~(\ref{eq:fwhalf}) the identity
\begin{equation}\label{eq:fubini-sum}
\sum_{j=0}^m\sum_{i=j}^{j+m} a(i,j) = \sum_{j=0}^m\sum_{\ell=0}^{m} a(\ell+j,\ell)\,,
\end{equation}
and recalling the definition~(\ref{eq:yhatl}), one easily gets
\begin{equation} \label{eq:Fx-nice}
F(x) = -\sum_{j=0}^m\widehat\by_{j}\log x_j +  \Big(\sum_{j=0}^m x_j\Big)^2\,.
\end{equation}

\noindent
The partial derivatives of $F$ immediately follow from Equation~\eqref{eq:Fx-nice} as
\[
\frac{\partial F}{\partial x_j}= -\frac{\widehat \by_j}{x_j}+2\sum_{\ell=0}^{m}x_\ell\,,\qquad j=0,\dots, m\,.
\]
Setting $\frac{\partial F}{\partial x_j}=0$ gives
\[
x_j^\star=\frac{\widehat\by_j}{2\sum_{\ell=0}^{m}x_\ell^\star}\,\raisepunct{,} \qquad j=0,\dots, m\,,
\]
and hence by summation
\begin{equation}\label{eq:normal-conv}
\Big(\sum_{j=0}^m x_j^\star\Big)^2=\half\sum_{j=0}^m\widehat\by_j = \sum_{i=0}^{2m} y_i\,.
\end{equation}
where, to prove the last identity, it is sufficient to observe that equation~(\ref{eq:yhatl}) defines $\widehat\by_j$ as the sum of the $j$-th subdiagonal and $j$-th column of the matrix $\by\in \bby$. This completes the proof of~(\ref{eq:xstarhalf}).
To prove the Pythagorean identity~(\ref{eq:pyth-2}) it is convenient to prove that $\ii(\by||\bw)-\ii(\by||\bw^\star)=\ii(\bw^\star||\bw)$, which is equivalent to
$$
\sum_{j=0}^m\sum_{i=j}^{j+m} \by_{ij}\log\frac{x_{i-j}^\star x_j^\star}{x_{i-j}x_j} = \sum_{j=0}^m\sum_{i=j}^{j+m} x_{i-j}^\star x_j^\star \log\frac{x_{i-j}^\star x_j^\star}{x_{i-j}x_j}\,\raisepunct{.}
$$
The last identity is easily verified by direct substitution of~(\ref{eq:xstarhalf}) and~(\ref{eq:yhatl}) to express $x^\star_j$, and using the identity~(\ref{eq:fubini-sum}). Again, as a byproduct of the Pythagorean identity one gets that $\bw^\star$ is indeed a minimizer for Problem \ref{problemwhalf}.
\end{proof}

\begin{remark}
Problem~\ref{problemwhalf} admits an interesting interpretation as a symmetric (constrained) rank one approximation of a given nonnegative matrix. We introduce the square matrices $\overline Y, \overline W \in \mathbb R^{(m+1)\times (m+1)}$, as `rectifications' of the $\by$ and $\bw$ matrices, defined as
$$
\overline Y_{ij}=\by_{i+j,j}\,, \qquad \overline W_{ij}=W_{i+j,j}=x_i x_j\,.
$$
Problem~\ref{problemwhalf} can be rephrased as
$$
\min_{x\in \mathbb R^{m+1}_+} D(\overline Y \mid\mid x x^\top)\,,
$$
whose solution is attained at
$$
x_i^\star = \frac{1}{2}\, \frac{\overline Y_{\cdot\,i}+\overline Y_{i\,\cdot}}{\sqrt{\sum_{ij}\overline Y_{ij}}}\,.
$$
In the probabilistic case ($\sum_{ij} \overline Y_{ij}=1$), the interpretation is that the best approximation of a two-dimensional distribution ($\overline Y$) by an i.i.d.\ product distribution ($x x^{\top}$) is attained at $x^\star$ equal to the average of the row and column marginals of $\overline Y$.
\end{remark}

\begin{remark} In the next section, when considering Problem~\ref{problemwhalf}, the given $\by\in\bby$ will always exhibit symmetry~(\ref{eq:sym-y}). When this is the case, Equation~(\ref{eq:xstarhalf}) for the optimal $x^\star$ simplifies considerably. Indeed, under symmetry~(\ref{eq:sym-y}), Equation~(\ref{eq:yhatl}) becomes
\begin{equation*}\label{eq:yhatlsimple}
\widehat\by_j=\sum_{\ell=0}^m\by_{\ell+j,\ell} + \sum_{i=j}^{j+m}\by_{ij}= 2 \sum_{i=j}^{j+m}\by_{ij}\,, \qquad j=0,\dots, m \,,
\end{equation*}
Equation~(\ref{eq:xstarhalf}) then reduces to
\begin{equation}\label{eq:xstarhalfsimple}
x_j^\star=\frac{1}{c}\,\sum_{i=j}^{j+m}\by_{ij} 
=\frac{1}{c}\,\sum_{\ell=0}^{m}\by_{\ell+j,j}\,,
\end{equation}
where
\begin{equation}\label{eq:def-c}
c:=\sqrt{\textstyle{\sum_{i=0}^{2m}y_i}}= \sum_{j=0}^m x_j^\star\,.
\end{equation}
\end{remark}
The connection between the original Problem~\ref{problemhalf} and the two lifted minimization problems is explained in the next proposition.

\begin{proposition}
The minimum of the original Problem~\ref{problemhalf} coincides with the double minimization Problems~\ref{problemyhalf} and~\ref{problemwhalf}, i.e.\
\[
\min_{x\in\rr^{m+1}_+}\ii(y||x* x)=\min_{\by\in\bby,\bw\in\bbw}\ii(\by||\bw).
\]
\end{proposition}
\begin{proof}
For given $x\in\rr^{m+1}_+$, with corresponding $\bw\in\bbw$, and $\by\in\bby$ consider the optimizers $\by^\star$ and $\bw^\star$ from Lemmas~\ref{lemmayhalf} and \ref{lemmawhalf} and recall Equation~\eqref{eq:fallback}. Then
$\ii(\by||\bw)\geq \ii(\by^\star||\bw)=\ii(y|x* x)\geq \min_x\ii(y||x* x)$, where the use of the minimum is justified by Proposition~\ref{proposition:exist}. Taking the joint minimum on the left hand side over $\by$ and $\bw$, justified by the just cited lemmas, leads to $\min_{\by,\bw}\ii(\by||\bw)\geq \min_{x}\ii(y||x* x)$.
Conversely, for given $x\in\rr^{m+1}_+$ with corresponding $\bw\in\bbw$, recalling again~\eqref{eq:fallback}, one obtains
\[
\ii(y||x* x)= \ii(\by^\star||\bw)= \min_\by\ii(\by||\bw)\geq \min_\bw\min_\by\ii(\by||\bw)\,,
\]
which, taking the minimum $x$, shows that $ \min_{x}\ii(y||x* x)\geq \min_{\by,\bw}\ii(\by||\bw)$, thus concluding the proof.
\end{proof}

\section{The algorithm}\label{section:algo}

This section is the core of the paper. It contains an algorithm aiming at finding a minimizer of Problem~\ref{problemhalf}, which we know to exist in view of Proposition~\ref{proposition:exist}, and an analysis of its behavior.

\subsection{Construction of the algorithm, basic properties}

Starting at an initial $\bw^0\in \bbw$ and combining the two partial minimization problems, one produces a classic alternating minimization sequence,
\begin{equation}\label{eq:altminseq}
\cdots \, \bw^t\stackrel{1}{\longrightarrow} \by^t\stackrel{2}{\longrightarrow}  \bw^{t+1}\stackrel{1}{\longrightarrow}  \by^{t+1} \, \cdots,
\end{equation}
where the superscript $t\in\mathbb N$ denotes the iteration step. The arrow $\stackrel{1}{\longrightarrow}$ denotes the  partial minimization Problem~\ref{problemyhalf}, the matrix at the tail of the arrow is the given input, and the matrix at the head $\by^t=\by^\star(\bw^t)$, is the optimal solution.
The meaning of $\stackrel{2}{\longrightarrow}$ is analogous, and represents the partial minimization Problem~\ref{problemwhalf}, and $\bw^{t+1}=\bw^\star(\by^t)$.
Note that, at each iteration, $\bw^t$ is completely specified by the fixed data $y$ and by the vector $x^t=(x_0^t,\dots, x_m^t) \in\mathbb R_+^{m+1}$. An iterative algorithm for the minimization Problem~\ref{problemhalf}, solely in terms of $x^t$, can be extracted from the sequence~(\ref{eq:altminseq})
as it immediately follows combining Lemmas~\ref{lemmayhalf} and~\ref{lemmawhalf}.
The update equation, say $x^{t+1}=I(x^t)$, is given below.
\begin{algorithm}\label{algorithm:half}
Starting from an arbitrary vector $x^0\in \mathbb R^{m+1}_+$ the update equation $x^{t+1}=I(x^t)$ is given componentwise by
\begin{equation} \label{eq:algohalf}
x^{t+1}_j= x^t_{j}\, \frac{1}{c} \sum_{\ell=0}^{m}\frac{x^t_\ell \,y_{\ell+j}}{(x^t* x^t)_{\ell+j}}\,, \qquad j=0,\dots, m\,.
\end{equation}
\end{algorithm}
To verify Equation~(\ref{eq:algohalf}) it is enough to shunt the $\by^t$ step in the chain~(\ref{eq:altminseq}) and concatenate directly $\bw^t$ to $\bw^{t+1}$. Starting with Equation~(\ref{eq:xstarhalfsimple}) and recalling the expression of $\by^t$ given by Equation~(\ref{eq:ystarhalf}) one has
\begin{equation}\label{eq:update}
x^{t+1}_j = \frac{1}{c}\,\sum_{\ell=0}^{m}\by^t_{\ell+j,j} = \frac{1}{c} \, \sum_{\ell=0}^{m}\frac{x^t_\ell x^t_{j}}{(x^t* x^t)_{\ell+j}} y_{\ell+j}\,,
\end{equation}
which coincides with~(\ref{eq:algohalf}).

\begin{remark}
Application of Algorithm~\ref{algorithm:half} to Example~\ref{expl:trivial-half-case} gives the exact solution in one step, starting from any initial $x^0_j>0$, as is easily verified. This is an exceptional case.
\end{remark}
The portmanteau proposition below summarizes some useful properties of the algorithm.

\begin{proposition}\label{proposition:properties}
The iterates $x^t,\, t\ge 0$, of Algorithm~\ref{algorithm:half} satisfy the following properties.
\begin{enumerate}[itemsep=-.1em,topsep=-.2em,label=(\roman{*}),labelsep=0.9em]
\item\label{item:hpositive}
If $x^0>0$ componentwise, then $x^t>0$ componentwise, for all $t>0$.
\item\label{item:simplex}
$x^t$ belongs to the simplex $\mathcal{S}=\{x\in\rr^{m+1}_+: \sum_{i=0}^mx_i=c\}$ for all $t>0$.
\item\label{item:Wt+1}
$\ii(y||x^t* x^t)$ decreases at each iteration, in fact one has
\begin{equation}\label{eq:gain1}
\ii(y||x^t* x^t) - \ii(y||x^{t+1}* x^{t+1}) = \ii(\by^t||\by^{t+1}) + \ii(\bw^{t+1}||\bw^t)\ge 0\,,
\end{equation}
and, as a corollary,  $\ii(\bw^{t+1}||\bw^t)$ vanishes asymptotically.
\item
If $y=x^t* x^t$ then $x^{t+1}=x^t$, i.e.\ perfect matches are fixed points of the algorithm.
\item
The update equation~(\ref{eq:algohalf}) can be written in the form
\begin{equation} \label{eq:algo-alt}
x_j^{t+1} = x_j^{t} \Big(1 - \frac{1}{2c}\nabla_j \ii(x^t)\Big)\,.
\end{equation}
\item
If $\nabla_j \ii(x^t)=0$ then $x^{t+1}_j =x_j^t$, and if $\nabla \ii(x)=0$ then $x^{t+1}=x^t$, i.e.\ stationary points of $\ii(x)$ are fixed points of the algorithm.
\item
If $\ii(x^t)$ is increasing (decreasing) in $x^t_j$, then $x^{t+1}_j<x^t_j$\, ($x^{t+1}_j>x^t_j$).
\end{enumerate}
\end{proposition}

\begin{proof}\mbox{}

\smallskip\noindent
(i)\, Obvious from \eqref{eq:algohalf}.

\smallskip\noindent
(ii)\, Consider the first equality in \eqref{eq:update}. Summing over $j$ gives
$$
\sum_{j=0}^mx^{t+1}_j = \frac{1}{c}\sum_{j=0}^m \sum_{\ell=0}^{m} \by^t_{\ell+j,j}=c\,,
$$
\noindent
in view of the two equalities in \eqref{eq:xstarhalfsimple} and \eqref{eq:def-c}.

\smallskip\noindent
(iii)\, Combining the Pythagorean identities~(\ref{eq:pyth-1}), (\ref{eq:pyth-2}) for the chain~(\ref{eq:altminseq}) one gets
\begin{equation*}
\ii(\by^t||\bw^t) = \ii(\by^t||\by^{t+1}) + \ii(\by^{t+1}||\bw^{t+1}) + \ii(\bw^{t+1}||\bw^t)\,,
\end{equation*}
from which Equation~\eqref{eq:gain1} follows applying Equation~(\ref{eq:fallback}). The corollary is proved noting that the decreasing sequence  $\ii(y||x^t* x^t)$ certainly has a limit therefore the LHS of the equation vanishes asymptotically and so do the terms on the RHS which are nonnegative for all $t>0$.

\smallskip\noindent
(iv)\, Under the assumption, \eqref{eq:algohalf} reduces to $x^{t+1}_j= x^t_j \frac{1}{c}\sum_{\ell=0}^m x^t_\ell =x^t_j$ in view of \ref{item:simplex}.

\smallskip\noindent
(v)\, From \eqref{eq:gradj} one gets
$\nabla_j \ii(x) = -2 \,\sum_{\ell=0}^m  \frac{x_\ell\, y_{\ell+j}}{(x*  x)_{\ell+j} } +2\sum_{\ell=0}^m x_\ell$, and recalling that $x^t\in \mathcal S$ it follows $\nabla_j \ii(x^t) = -2 \,\sum_{\ell=0}^m \frac{x^t_\ell\, y_{\ell+j}}{(x^t*  x^t)_{\ell+j} } +2\,c$.
Hence, the update equation~(\ref{eq:algohalf}) can be written as in \eqref{eq:algo-alt}.

\smallskip\noindent
(vi), (vii) follow immediately from (v).
\end{proof}

\subsection{Convergence analysis}

The aim of this section is to investigate the behaviour of Algorithm~\ref{algorithm:half} for large values of the iteration index $t$. We start with a technical  lemma.
\begin{lemma}\label{lemma:tt1}
For the iterates $x^t$ and their corresponding $\bw^t$ it holds that
\begin{enumerate}[itemsep=-.1em,topsep=-.2em,label=(\roman{*}),labelsep=0.9em]
\item\label{item:IWt}
$\ii(\bw^{t+1}||\bw^t)=2 c\, \ii(x^{t+1}||x^t)$,
\item
$\sum_i|x^{t+1}_i-x^t_i|\leq\big(\ii(\bw^{t+1}||\bw^t)\big)^{1/2}$,
\item
$\lim_{t\rightarrow \infty} \ii(x^{t+1}||x^t) = 0$, and hence $\sum_i|x^{t+1}_i-x^t_i|\to 0$.
\end{enumerate}
\end{lemma}
\begin{proof}
To prove (i) a direct computation gives
\begin{align}
\ii(\bw^{t+1}&||\bw^t)=\sum_{j=0}^m\sum_{i=j}^{j+m} \bigg(\bw_{ij}^{t+1} \log \frac{\bw_{ij}^{t+1}}{\bw_{ij}^t}-\bw^{t+1}_{ij}+\bw^t_{ij}\bigg) \nonumber\\
&= \sum_{j=0}^m \sum_{i=j}^{j+m} x_{i-j}^{t+1}x_j^{t+1} \log \frac{x_{i-j}^{t+1}x_j^{t+1}}{x^t_{i-j}x^t_j}
= \sum_{j=0}^m \sum_{\ell=0}^m x_\ell^{t+1}x_j^{t+1} \log \frac{x_\ell^{t+1}x_j^{t+1}}{x^t_\ell x^t_j} \nonumber\\
&= 2 \bigg(\sum_{\ell=0}^m x_\ell^{t+1}\bigg) \sum_{j=0}^m x_j^{t+1} \log\frac{x_j^{t+1}}{x^t_j} = 2 c\, \ii(x^{t+1}||x^t)\,, \label{eq:IWIx}
\end{align}
where the last identity follows from Equation~(\ref{eq:normal-conv}).

To prove (ii) recall Pinsker's inequality which states, for probability vectors $p, q$, that $\sum_i|p_i-q_i|\leq \left(2\ii(p||q)\right)^{1/2}$. The iterates $x^t$ and $x^{t+1}$ are not probability vectors in general, but both belong to the simplex $\mathcal{S}$ therefore, by an easy corollary to Pinsker's inequality, $\sum_i|x^{t+1}_i-x^t_i|\leq \left(2c\,\ii(x^{t+1}||x^t)\right)^{1/2}$, from which, by direct application of (i), one gets (ii).

Finally, (iii) descends from the fact that $\ii(\bw^{t+1}||\bw^t)$ vanishes asymptotically, as proved by the corollary to Equation (\ref{eq:gain1}), and therefore, applying (i) again, so does $\ii(x^{t+1}||x^t)$ and by Pinsker's inequality also $\sum_i|x^{t+1}_i-x^t_i|$.
\end{proof}
The existence of limit points of the sequence $(x^t)$ of the iterates of the algorithm is obvious as all $x^t$ belong to the simplex $\mathcal{S}$, see Proposition~\ref{proposition:properties}, which is a compact set. Note that the sequence $(x^t)$ depends on the initial point $x^0$. Changing $x^0$ the sequence $(x^t)$ changes and so do, in general, its limit points. To avoid a cluttered notation the dependence of the limit points on $x^0$ will not be evidenced. We continue with establishing some properties of the limit points of $x^t$.

\begin{lemma} \label{lem:conv-2} If $x^\infty$ is a limit point of the sequence $(x^t)$ then it is a fixed point of the algorithm, i.e.
$$
x^\infty = I(x^\infty)\,.
$$
\end{lemma}
\begin{proof}
Let $x^\infty$ be a limit point of the $x^t$. The map $x^{t+1}=I(x^t)$, given componentwise in~(\ref{eq:algohalf}), is continuous. Likewise the I-divergence $\ii(u||v)$ is jointly continuous in $(u, v)$ for all $v>0$. It follows that $\ii(I(x^\infty)||x^\infty)$ is a limit point of the $\ii(x^{t+1}||x^t)$ which, by Lemma~\ref{lemma:tt1} (iii), vanishes asymptotically implying that $\ii(I(x^\infty)||x^\infty)=0$, which yields $x^\infty=I(x^\infty)$, i.e.\ $x^\infty$ is a fixed point of the algorithm.
\end{proof}


\begin{proposition}\label{prop:limitpoint}
The I-divergence $\ii(y||x^\infty * x^\infty)$ is constant over the set of all limit points $x^\infty$ of $(x^t)$.
\end{proposition}

\begin{proof}
Iteration of \eqref{eq:gain1} gives for $t\leq T$
\begin{equation}\label{eq:gain2}
\ii(y||x^t* x^t) - \ii(y||x^{T}* x^{T}) = \sum_{k=t}^{T-1}\Big(\ii(\by^{k+1}||\by^{k}) + \ii(\bw^{k+1}||\bw^k)\Big)\,.
\end{equation}
Suppose that the  $x^T$ converge along a subsequence to $x^\infty$. Then we also have
\begin{equation}\label{eq:gain3}
\ii(y||x^t* x^t) - \ii(y||x^{\infty}* x^{\infty}) = \sum_{k=t}^{\infty}\Big(\ii(\by^{k+1}||\by^{k}) + \ii(\bw^{k+1}||\bw^k)\Big)\,.
\end{equation}
Suppose $x'$ is another limit point and $x^t$ converges to $x'$ along a suitable subsequence indexed by $t^\prime$. Taking the limit for $t = t'\to\infty$ in Equation (\ref{eq:gain3}), one sees that the RHS vanishes, whereas the LHS gives $\ii(y||x'* x')-\ii(y||x^\infty * x^\infty)$, which is thus zero.
\end{proof}

\begin{remark}\label{remark:cc}  Proposition~\ref{prop:limitpoint} makes it clear that all limit points of $x^t$ are equivalent, in the sense that their autoconvolutions have the same informational distance to the target $y$  in Problem \ref{problemhalf}. In particular, if one limit point is a minimizer, so are all other limit points.

\medskip\noindent
One can show that the set of limit points of the sequence $(x^t)$ is compact and connected. Compactness follows from Proposition~\ref{prop:limitpoint} (the set of limit points is closed and contained in the simplex $\mathcal{S}$, hence bounded), whereas connectedness is essentially a consequence of Lemma~\ref{lemma:tt1}. A similar statement can be found in \cite{cover1984}.
\end{remark}

\begin{proposition}\label{prop:kt}
Limit points of the sequence $(x^t)$ are Kuhn-Tucker points of the minimization Problem~\ref{problemhalf}.
\end{proposition}

\begin{proof}
Recall the version of the update equation of the algorithm as in \eqref{eq:algo-alt}.
By Lemma~\ref{lem:conv-2}, if $x^\infty$ is a limit point of the  $x^t$ then it is a fixed point of the algorithm, and Equation~(\ref{eq:algo-alt}) reduces to
\begin{equation*} \label{eq:algo-alt-stat}
x_j^\infty = x_j^\infty \Big(1 - \frac{1}{2c}\nabla_j\ii(x^\infty)\Big)\,,
\end{equation*}
showing that, if $x^\infty_j>0$ then $\nabla_j\ii(x^\infty)=0$. To complete the verification that $x^\infty$ satisfies the Kuhn-Tucker conditions for $\ii(x)$ one has to check that if $x^\infty_j=0$ then $\nabla_j\ii(x^\infty)\geq 0$. So we proceed with investigating limit points on the boundary.
For a given initial condition $x^0$, let $(x^t)$ be the sequence of iterates of the algorithm and define
$O=\{x\in\rr^{m+1}_+: \nabla_j\ii(x)<0\}$. Put $L_0=0$ and let $U_0=\inf\{t> 0: x^t \in O^c\}$. If $U_0=\infty$, then all $x^t$ belong to $O$ and the $x_j^t$ form an increasing  sequence in view of Equation~\eqref{eq:algo-alt}, so certainly all $x^t_j>x^0_j>0$ and a limit point with $x_j^\infty=0$ cannot occur.
If $U_0$ is finite, we put $L^1=\inf\{t>U_0: x^t\in O\}$. If $L_1=\infty$, then $x^t\in O^c$ for all $t\geq U_0$, so the $x^t_j$ form a decreasing sequence, converging to some $x^\infty_j\ge 0$. With $\nabla_j\ii(x^t)\geq 0$ for all $t$, then necessarily also $\nabla_j\ii(x^\infty)\geq 0$, hence $x^\infty$ satisfies the Kuhn-Tucker conditions. In case $L^1<\infty$ continue by alternating definitions, $U_1=\inf\{t > L_1: x^t\in O^c\}$, $L_2=\inf\{t > U_1: x^t \in O\}$, etc. As soon as some $L_k$ or $U_k$ is infinite, we are in either of the situations just described and  in a limit point one necessarily has $x_j^\infty>0$ or $x^\infty_j\ge 0$ and $\nabla_j\ii(x^\infty)\geq 0$ satisfying the Kuhn-Tucker conditions.

As a last case, we investigate what happens if all $L_k$ and $U_k$ are finite and the interest is in possible boundary limit points $x^\infty$ with $x^\infty_j=0$. Observe that for $t$ between the $L_k$ and $U_k$ the $x^t_j$ are increasing and for $t$ between the $U_k$ and $L_{k+1}$ the $x^t_j$ are decreasing. More precisely, for $L_k\leq t <U_k$ it holds that $x^{t+1}_j\leq x^t_j$ and for $U_k\leq t <L_{k+1}$ it holds that $x^{t+1}_j> x^t_j$. In particular $x^{L_k}_j\leq x^{L_k-1}_j$ and $x^{L_k}_j< x^{L_k+1}_j$, hence $x^{L_k}_j$ is a local minimum of the $x^t_j$. Suppose that $x^\infty$ is a limit point, with $x^\infty_j=0$. Then we have to consider the liminf of the $x^t_j$, which coincides with the liminf of the $x^{L_k}_j$. But, by Lemma~\ref{lemma:tt1}, then also $x^{L_k-1}_j$ converges along a subsequence to the same liminf, and in these points one has $\nabla_j\ii(x^{L_k-1})\geq 0$. Hence along any convergent subsequence of the $x^t$ with $\liminf x^t_j=0$, one necessarily has $\nabla_j\ii(x^{\infty})\geq 0$. As a side remark, in this last case, since $\nabla_j\ii(x^{L_k})< 0$ for all $k$ one gets in fact $\nabla_j\ii(x^{\infty})= 0$.
\end{proof}

\subsection{Convergence properties, further considerations}

All empirical examples suggest that the iterates of Algorithm~\ref{algorithm:half} converge to a limit. Although a full proof cannot be given, a number of considerations make this result plausible, also from a theoretical point of view.

On a technical note, in order to prove that the algorithm converges, one would need to show that $\ii(x^\infty||x^t)$ is decreasing in $t$,  for any limit point $x^\infty$. The proof of this property would go along the arguments of Lemma~A.1 of \cite{Vardietal1985} or Lemma~24 in \cite{finessospreij2015ieeeit}, if one could prove that, in our notation, $\ii(\bw^\infty||\bw^t)\leq  c\, \ii(x^\infty||x^t)$. Unfortunately it is only possible to prove the looser inequality $\ii(\bw^\infty||\bw^t)\le 2c\, \ii(x^\infty||x^t)$. The factor 2 essentially appears as a consequence of the `quadratic nature in $x$' of the autoconvolution terms $(x* x)_i$ whereas terms of type $(u* x)_i$, appearing in the context of e.g.\ \cite{finessospreij2015ieeeit} or \cite{finessospreij2019automatica}, are linear in $x$. Consequently one cannot conclude that the $x^t$ converge to a global minimizer.  For completeness we present, in Proposition~\ref{proposition:decrease}, the proof of convergence of the algorithm under the proviso, empirically satisfied in all cases, that $\ii(x^\infty||x^t)$ decreases in $t$. In the simulations Section~\ref{section:numerics} we shall see an example where convergence of the $x^t$ occurs, but not to a global minimizer of $\ii(x)$.

\begin{proposition}\label{proposition:decrease}
Let $x^\infty$ be a limit point of the sequence $(x^t)$ and assume that $\ii(x^\infty||x^t)$ is decreasing in $t$. Then $x^t$ converges to $x^\infty$, which is the unique limit point of $x^t$.
\end{proposition}
\begin{proof}
By Proposition~\ref{proposition:properties}\ref{item:simplex} the  $x^t$ belong to $\mathcal S$ and therefore, along some subsequence, $x^{t_k}\rightarrow x^\infty$, for some limit point $x^\infty\in \mathcal{S}$. By continuity $\ii(x^\infty||x^{t_k})\rightarrow 0$. On the other hand, as the divergences $\ii(x^\infty||x^t)$ are decreasing, it must hold that $\ii(x^\infty||x^t)\rightarrow 0$. Using Pinsker's inequality as in the proof of Lemma~\ref{lemma:tt1}, $\sum_i|x^\infty_i-x^t_i|\leq \left(2c\,\ii(x^\infty||x^t)\right)^{1/2}$, one concludes that $x^t \rightarrow x^\infty$, and hence that $x^\infty$ is the unique limit point.
\end{proof}
Next to the empirically observed behavior in Proposition~\ref{proposition:decrease}, we present another argument for convergence based on an element of Morse theory, for which we need the Hessian of the criterion $\ii(x)$.
Differentiate
$\frac{\partial\ii(x)}{\partial x_j}$
as given by \eqref{eq:gradj} w.r.t.\ $x_i$ to get
\[
H_{ij}(x):=-2\frac{y_{i+j}}{(x* x)_{i+j}}+4\sum_{l=0}^{2m}\frac{y_{l+j}}{(x* x)_{l+j}^2}x_lx_{l+j-i}+2.
\]
Note that effectively the index $l$ in the summation runs from $\max\{i-j,0\}$ to $m+\min\{i-j,0\}$, because of our convention $x_\ell=0$ for $\ell<0$ or $\ell >m$. 
The expression for $H_{ij}(x)$ can be rewritten as
\[
H_{ij}(x):=-2\frac{y_{i+j}}{(x* x)_{i+j}}+4\sum_{k=0}^{2m}\frac{y_{k}}{(x* x)_{k}^2}x_{k-j}x_{k-i}+2,
\]
with the same conventions as for the previous display. Effectively the index $k$ in the summation runs from $\max\{i,j\}$ to $m+\min\{i,j\}$.

We introduce the matrices $S^{(k)}\in\mathbb R^{(m+1)\times(m+1)}$, for $k\in\{0,\ldots,2m\}$, defined by $S^{(k)}_{ij}=\delta_{k,i+j}$ for  $i,j\in\{0,\dots,m\}$, where the $\delta$'s are the Kronecker $\delta$'s. Let furthermore $x=(x_0,\dots,x_m)^\top$ and $\xi_k=S^{(k)}x$. Define $P(x)\in\mathbb R^{(m+1)\times(m+1)}$ with elements $P_{ij}(x)=4\sum_{k=0}^{2m}\frac{y_{k}}{(x* x)_{k}^2}x_{k-j}x_{k-i}$, then one can write
\[
P(x)=4\sum_{k=0}^{2m}\frac{y_{k}}{(x* x)_{k}^2}\xi_k\xi_k^\top.
\]
Note that, if $x_0> 0$, the $\{\xi_k\}_{k=0}^m$ form a basis of $\rr^{m+1}$, therefore if $y_k>0$ for $k\in[0,m]$, the matrix $P(x)$ is strictly positive definite. Alternatively, if $x_m>0$, the $\{\xi_k\}_{k=m}^{2m}$ also form a basis of $\rr^{m+1}$  and again, if $y_k >0$ for $k\in[m,\,2m]$, the matrix $P(x)$ is  strictly positive definite.
Furthermore, let $Q(x)\in\mathbb R^{(m+1)\times(m+1)}$, with elements $Q_{ij}(x)=2-2\frac{y_{i+j}}{(x* x)_{i+j}}$. Hence the Hessian $H(x)$ satisfies
\[
H(x)=P(x)+Q(x).
\]
Note that $Q(x)$ vanishes if $y_i=(x* x)_i$, for all $i\in[0,2m]$, i.e.\ in the exact model case, making $H(x)$ strictly positive definite.
To find a useful expression of the Hessian in the general case introduce the matrix $R(x)\in\mathbb R^{(m+1)\times(m+1)}$ with elements $R_{ij}(x)=\frac{y_{i+j}}{(x* x)_{i+j}}$, and note that $Q(x)=2\left(\one\one^\top-R(x)\right)$, then
\[
H(x)=P(x)+2\,\left(\one\one^\top-R(x)\right)\,,
\]
moreover the gradient $\nabla \ii(x)$, written as a row vector, is
\[
\nabla\ii(x)=x^\top Q(x) = 2x^\top\left(\one\one^\top- R(x)\right)\,.
\]

Except in the special case of an exact model, it is not obvious that in an interior limit point $x^\infty$ of the algorithm the Hessian $H(x^\infty)$ is strictly positive definite. Even the weaker statement that $H(x^\infty)$ is non-singular is hard to prove, in spite of the rather explicit form of $H(x^\infty)$ and the fact that the gradient $\nabla\ii(x^\infty)$ vanishes. The relevance of  non singularity stems from the Morse lemma, Corollary~2.3 in \cite{milnor1963}, which states that, the interior critical points of a function where the Hessian in is non singular are isolated.

Let us now look at a boundary (local) optimizer $x^\star$ of $\ii(x)$. By the Kuhn-Tucker conditions if  $x^\star_j=0$  then $\nabla_j\ii(x^\star)\geq 0$, while if $x^\star_j>0$ then $\nabla_j\ii(x^\star)= 0$.
Write the boundary optimizer $x^\star$ as $x^\star=(\underline{x}^\star,0)$, possibly after a permutation of the coordinates, with all elements of $\underline{x}^\star$ strictly positive. We now look at optimization of $\ii(x)$ under the constraint that $x=(\underline{x},0)$, so of $\underline{\ii}(\underline{x}):=\ii(\underline{x},0)$. The optimizing $\underline{x}^\star$ is now an interior point of the restricted domain, hence the gradient vanishes, $\nabla\underline{\ii}(\underline{x}^\star)=0$.
The Hessian $\underline{H}(\underline{x}^\star)$ of $\underline{\ii}(\underline{x})$, is strictly positive definite, certainly non-singular and likely the same is true for $\underline{H}(\underline{x}^\infty)$ for any limit point $(\underline{x}^\infty,0)$ of $x^t$. The arguments underlying this are similar to the above, although it is  hard to give a proof.
Again by the Morse lemma, the critical points of $\underline{\ii}(\underline{x})$, which are now interior points of the restricted domain, will then be isolated.

\begin{proposition}
Let $x^0$ be a strictly positive starting point of the algorithm and let $L(x^0)$ be the set of interior limit points produced by the algorithm and assume that $H(x)$ is non-singular for all $x\in L(x^0)$. Then $L(x^0)$ is a singleton and thus the algorithm converges to a limit (possibly depending on the starting value $x^0$). The situation is analogous for boundary limit points.
In both cases the limit is a Kuhn-Tucker point.
\end{proposition}

\begin{proof}
By Remark~\ref{remark:cc}, the set $L(x^0)$ is connected. By the above  discussion the  interior limit points are isolated and the same holds for the limit points on the boundary. The combination of these two properties yields that $L(x^0)$ has to be a singleton, and hence there is convergence of $x^t$ to the (unique) limit. Its Kuhn-Tucker property follows from Proposition~\ref{prop:kt}.
\end{proof}

\begin{remark}
In the literature it is not uncommon to see situations where the limit points are isolated.
For instance, along different lines, in \cite{lange1984reconstruction} and \cite{lange1995globally} it is shown that in their setting the set of limit points of the iterates is finite, which is there a consequence of the maximization of a concave objective function. As the objective function in our minimization problem is not convex, their arguments cannot be taken over.
\end{remark}

\begin{remark}
In principle, the algorithm may produce different limit points, due to different initial values $x^0$. This has been observed in various numerical experiments. In fact, different starting values may either result in an interior limit or in a limit on the boundary, some of its coordinates are zero. The Kuhn-Tucker property was seen to be verified in these experiments.
\end{remark}

\noindent
To summarize the discussion of this section, it is very plausible that Algorithm~\ref{algorithm:half}, given a starting value, converges to  a limit. This conjecture is motivated by two considerations, for both of them there is ample numerical evidence. The first one is a decreasing criterion, of which Proposition~\ref{proposition:decrease} takes care, and the second is non-singularity of the Hessian in limit points. Yet, a formal proof of the conjecture is lacking and we have to content ourself with the Kuhn-Tucker property of limit points as in Proposition~\ref{prop:kt}.

\section{Numerical experiments}\label{section:numerics}

In this section we review the results of numerical experiments for three different data sets to illustrate the behaviour of Algorithm~\ref{algorithm:half}. For the first two data sets, with $m=25$ and $m=10$ respectively, we investigated whether the algorithm is capable of retrieving the true parameter vector $x$, when the data $y$ are actually generated by the autoconvolution $y=x* x$. In the third data set, with $m=10$, the data $y$ are randomly generated.

To evaluate the performance of the algorithm we have generated, for each data set, one figure comprising three or four graphs. In all of the figures the top graph shows, in distinct colors, the trajectories of the iterates of the components, $x^t_i$, plotted against the iteration number $t\in[1, T]$.

In the exact model case, Figures~\ref{fig:good_one_N_50_Nit_1000},
\ref{fig:bad_one_N_50_Nit_1000_Xtrue_given_IC_bad},
\ref{fig:good_one_N_20_Nit_100_Xtrue_rand},
the diamonds at the right end of the top graph show the true $x_i$ values. The second graph shows the superimposed plots of the data generating signal $x$, and of the reconstructed signal $x^T$, at the last iteration, both plotted against their component number $i=0, 1, \dots, m$. The third graph shows the decreasing sequence $\ii(y||x^t* x^t)$.  The fourth and last graph shows the superimposed plots of the data vector $y$ and of the reconstructed convolution $x^T* x^T$, at the last iteration, both against the component number $i=0,1,\dots 2m$.

Figures~\ref{fig:random_Y_N=20_Nit_100_run-1} and \ref{fig:random_Y_N=20_Nit_100_run-2}, relative to the randomly generated data set, contain only three graphs, as the graph of the data generating signal is meaningless in this case.

We have observed experimentally that the iterative algorithm always converges very fast. The precise features underlying the experiments are further detailed below. All figures are collected at the end of the paper.

\subsection{True autoconvolution systems}

For the first data set we have taken $m=25$. The components
of the true vector $x$ (the target values of the algorithm) have been randomly generated from a uniform distribution on the interval $[1, 11]$, and the data computed as true autoconvolutions $y=x* x$. The algorithm has been initialized at a randomly chosen strictly positive $x^0$, with components generated from a uniform distribution in the interval $[0.1,\, 0.2]$ and run for $T=1000$ iterations. Figures~\ref{fig:good_one_N_50_Nit_1000} and \ref{fig:bad_one_N_50_Nit_1000_Xtrue_given_IC_bad} show the results for two different runs (i.e.\ with the same true vector but different initial conditions) of the algorithm. In Figure~\ref{fig:good_one_N_50_Nit_1000} we see the desired behavior of the algorithm, the iterates converge to the true values and the divergence decreases to zero (because of the perfect match of $y=x* x$). This is the behavior that has been observed in a vast majority of numerical experiments of this kind. In Figure~ \ref{fig:bad_one_N_50_Nit_1000_Xtrue_given_IC_bad} we observe a different behavior. The iterates do not converge to the true values (see the second graph) and the divergence does not decrease to zero. On the other hand the convolution $x^T* x^T$ is always close to $y$ (see the fourth graphs of both figures).  In fact, the instance of running the algorithm that produced Figure~\ref{fig:bad_one_N_50_Nit_1000_Xtrue_given_IC_bad} produced iterates that converged to a non-optimal local minimum of the objective function $\ii(x)$. Indeed, we have verified that the gradient of $\ii(x)$ at the final iteration vanished, whereas the Hessian turned out to be strictly positive definite. The conclusion of these two experiments is that it is wise to run the algorithm for the same data $y$, and same $x$, with different initial conditions and select the outcome with the lowest divergence. For the present example, the lowest divergence is of course zero, but the conclusion is also valid for any instance with any data vector $y$.

The data set used to generate Figure~\ref{fig:good_one_N_20_Nit_100_Xtrue_rand} is again of the exact type, $y=x* x$, with $m=10$ and consequently a lower number of iterations, $T=100$. We see quick convergence of the algorithm, stabilization has already occurred at $t=30$. The general behavior is identical to that observed in Figure~\ref{fig:good_one_N_50_Nit_1000}.

\subsection{Approximation of arbitrary data}

For the third data set there is no true input signal $x$ such that $y=x* x$, rather the components of the data vector $y$, with $m=10$, have been randomly generated from a uniform distribution on the interval $[0.1,\, 2]$. Thus, here we deal with a genuine approximation problem.
Figures~\ref{fig:random_Y_N=20_Nit_100_run-1} and Figure~\ref{fig:random_Y_N=20_Nit_100_run-2} show the results of two runs of the algorithm, for $T=100$ iterations, and are relative to the same $y$ vector and different initial conditions $x^0$, both with components randomly generated from a uniform distribution in $[0.1,\, 0.2]$.
The aim is to find the vector $x$ which yields the best autoconvolutional approximation to $y$. Inspecting the figures we conclude that the algorithm quickly stabilises in both runs. The final values $x^T$ of the iterates and the final divergences $\ii(y||x^T* x^T)$ differ in the two runs, indicating that (at least) in the second case  (with divergence slightly higher than in the first case) the algorithm is trapped in a non-optimal local minimum. For the same $y$ several other runs have produced results that were nearly identical to those in Figure~\ref{fig:random_Y_N=20_Nit_100_run-1}, so we infer that this figure represents the optimal approximation of $y$. The observed behavior suggests again to run the algorithm with different initial conditions, possibly in parallel, and to select the best final approximation as the one with smallest divergence $\ii(y||x^T* x^T)$.

\bibliographystyle{plain}

\begin{figure}
\begin{center}
	\includegraphics[width=0.9\textwidth]{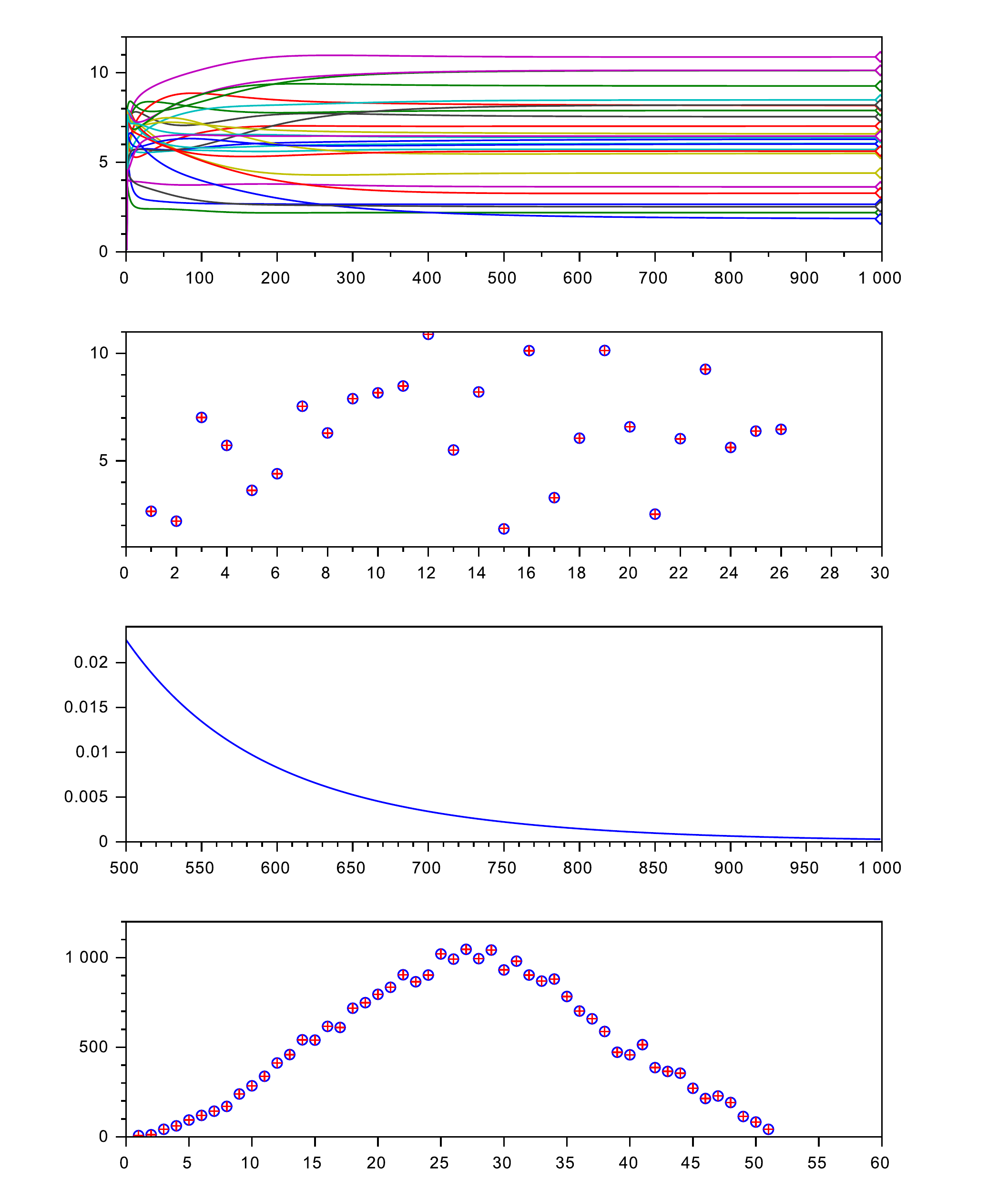}
	\vspace{5ex}
		\caption{True model, $m=25$ and $T=1000$. Top panel: $m+1$ components $x^t_i$ against iteration $t$; the diamonds at $T=1000$ are the true values $x_i$ to which the $x^t_i$ converge. Second panel: $x^T_i$ (plusses) and true values $x_i$ (circles) against $i$. Third panel: $\ii(y||x^t* x^t)$ against $t$. Fourth panel: $y_i$ (circles) and $(x^T* x^T)_i$ (plusses) against $i$.} \label{fig:good_one_N_50_Nit_1000}
\end{center}
\end{figure}

\vfill

\begin{figure}
\begin{center}
	\includegraphics[width=0.9\textwidth]{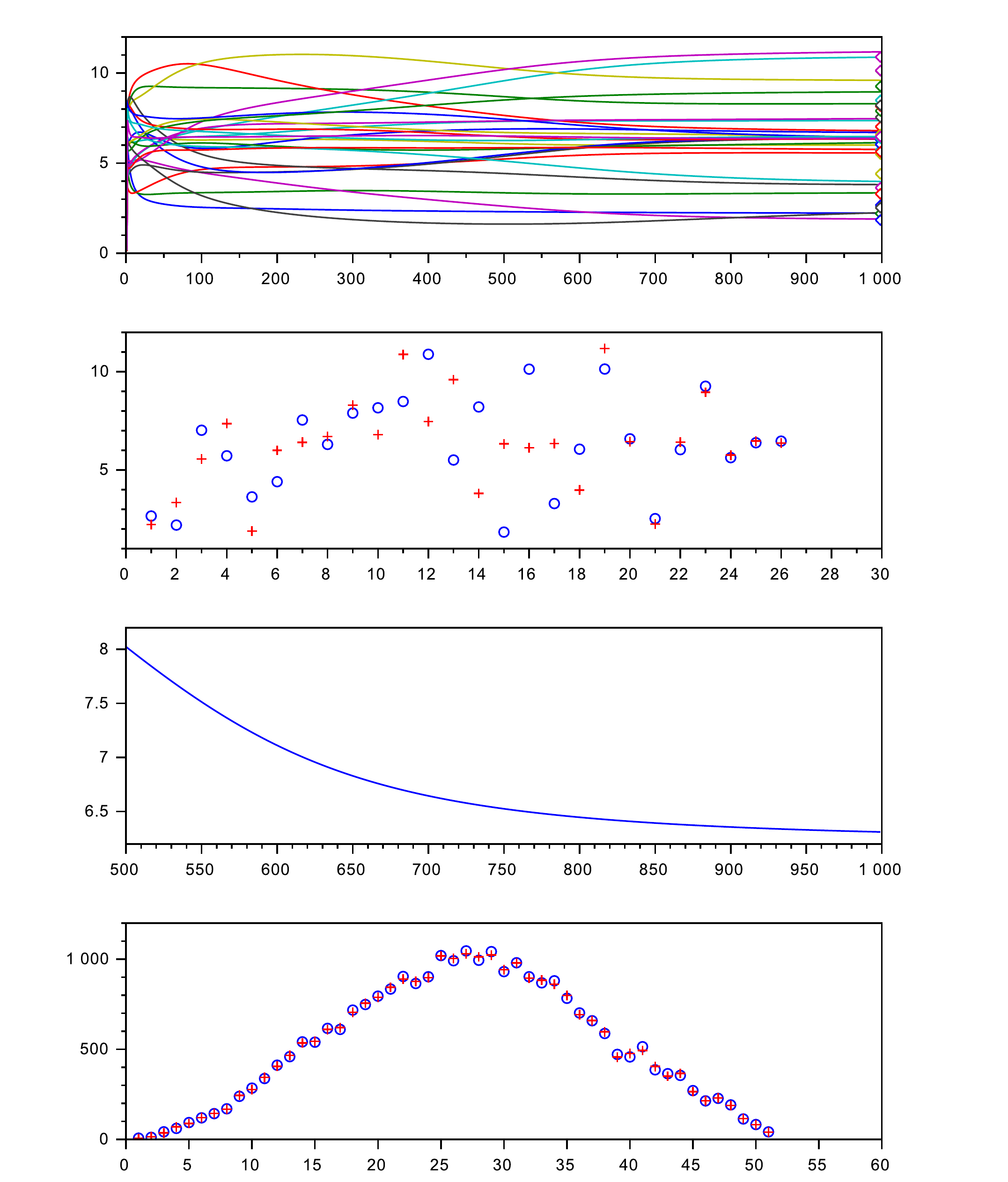}
	\vspace{5ex}
    \caption{The same data as in Figure~\ref{fig:good_one_N_50_Nit_1000}, with different initial conditions $x^0$. }
		\label{fig:bad_one_N_50_Nit_1000_Xtrue_given_IC_bad}
\end{center}
\end{figure}

\begin{figure}
\begin{center}
	\includegraphics[width=0.9\textwidth]{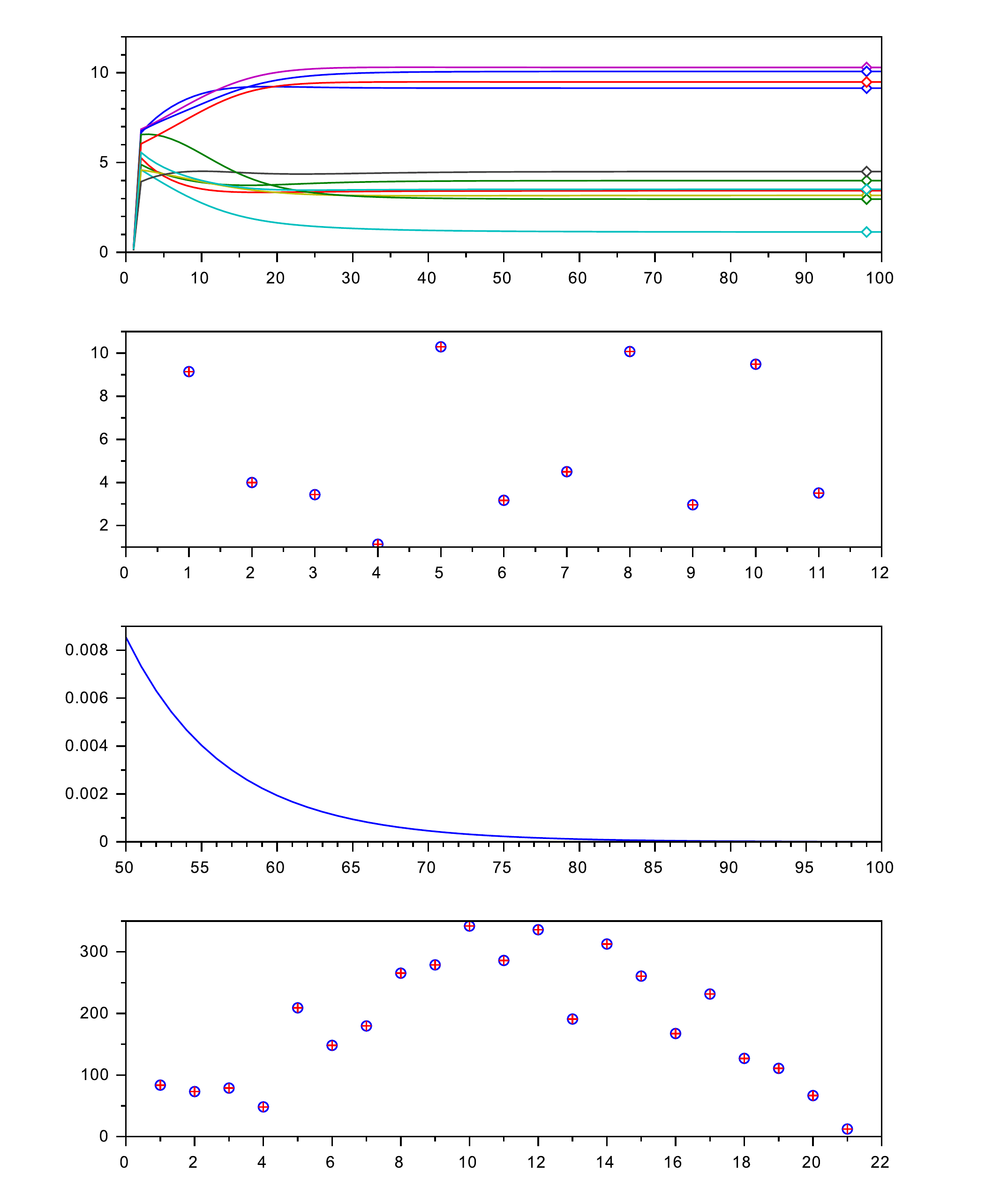}
	\vspace{5ex}
		\caption{A true model with $m=10$ and $T=100$. The panels are as in Figure~\ref{fig:good_one_N_50_Nit_1000} and the same conclusions can be drawn.}
		\label{fig:good_one_N_20_Nit_100_Xtrue_rand}
\end{center}
\end{figure}

\begin{figure}
\begin{center}
	\includegraphics[width=0.9\textwidth]{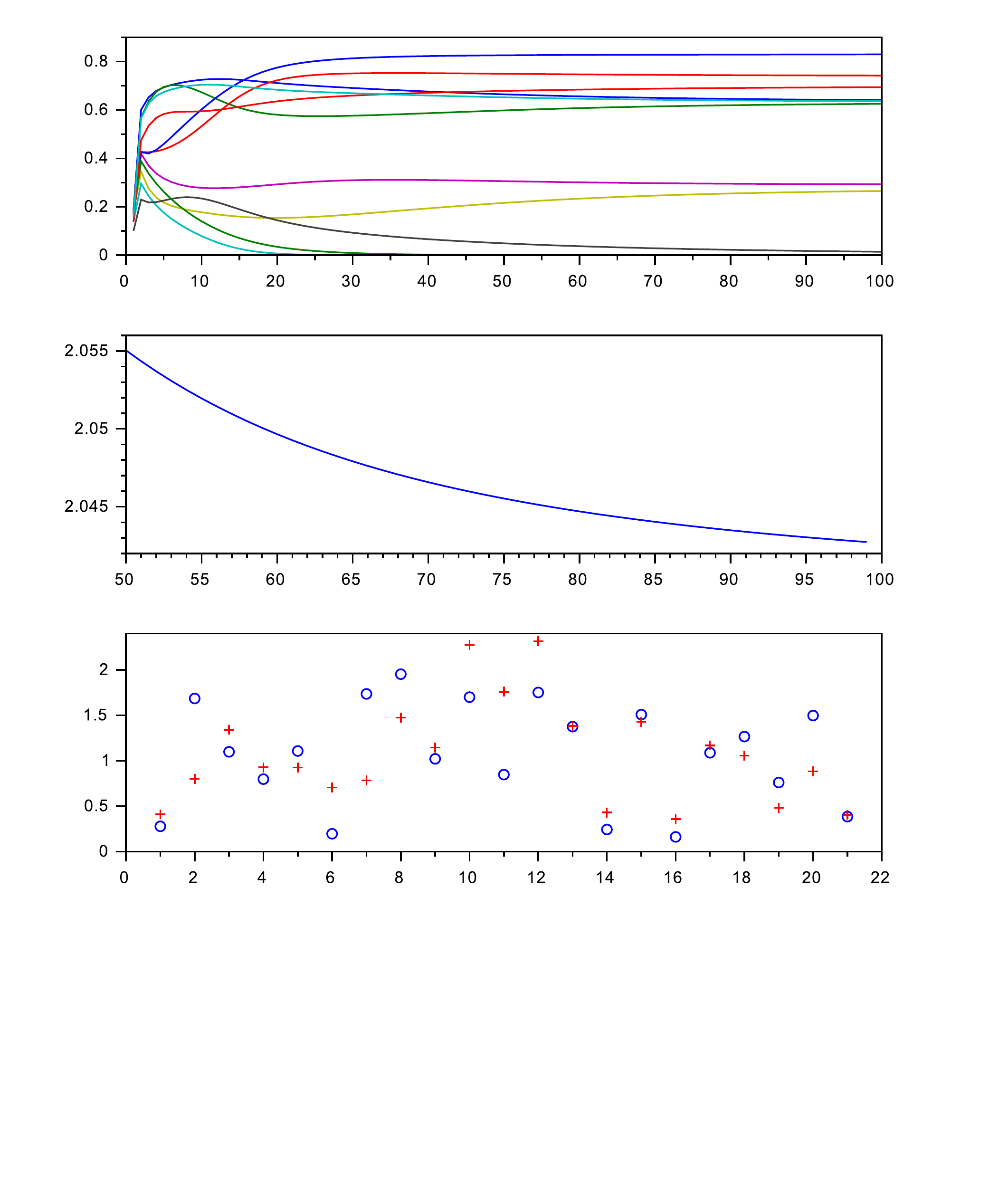}
	\vspace{-15ex}
		\caption{Randomly generated $y$, with $m=10$ and $T=100$. Top panel: components $x^t_i$ against iteration index $t$. Second panel:  $\ii(y||x^t* x^t)$ against $t$. Third panel: $y_i$ (circles) and final autoconvolutions $(x^T* x^T)_i$ (plusses) against $i$. Third panel: The values of the $y_i$ (circles) and the final autoconvolutions $(x^T* x^T)_i$ (plusses).}
		\label{fig:random_Y_N=20_Nit_100_run-1}
\end{center}
\end{figure}

\begin{figure}
\begin{center}
	\includegraphics[width=0.9\textwidth]{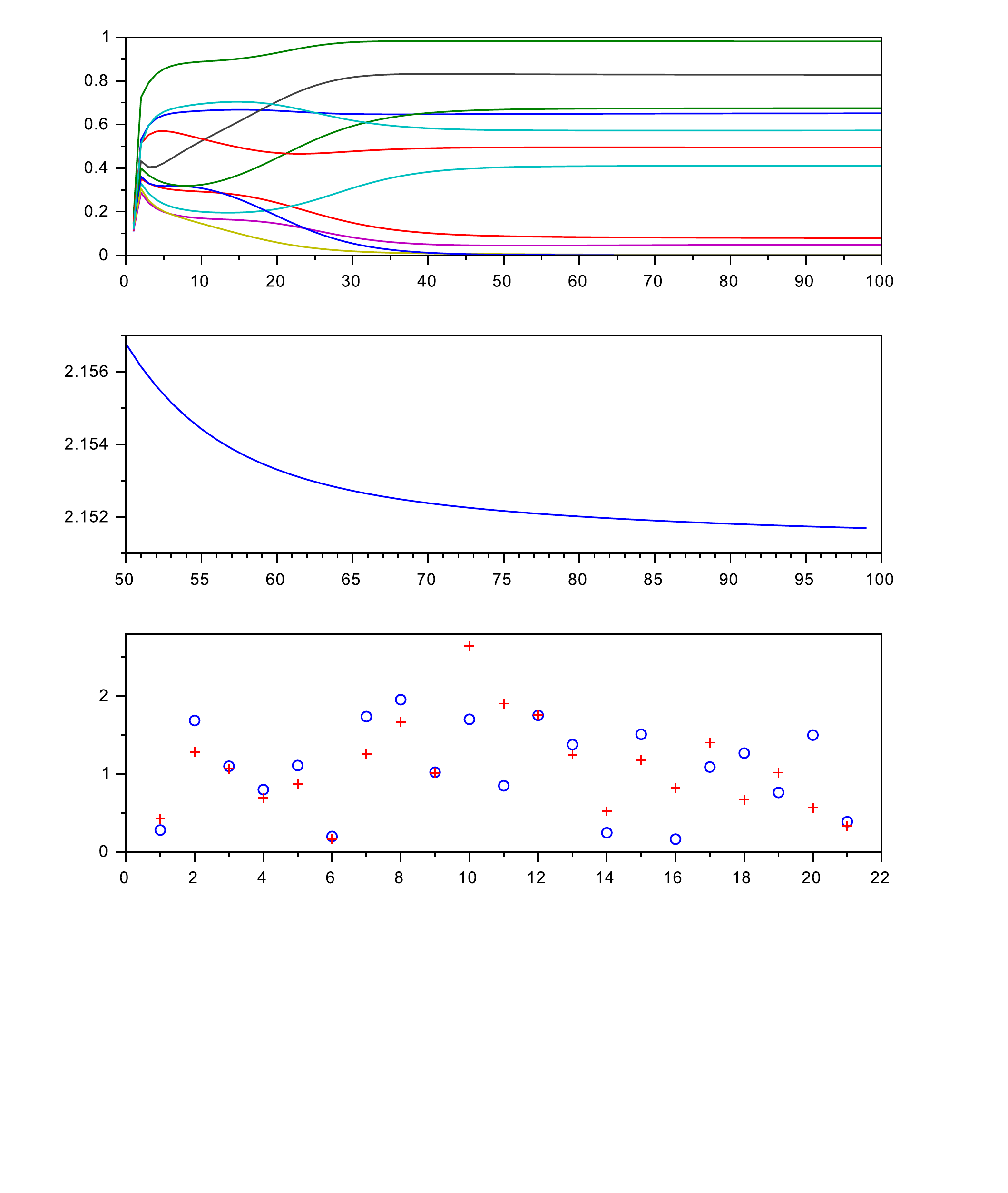}
	\vspace{-15ex}
		\caption{The same data as in Figure~\ref{fig:random_Y_N=20_Nit_100_run-1}, with different initial conditions $x^0$.}
		\label{fig:random_Y_N=20_Nit_100_run-2}
\end{center}
\end{figure}

\end{document}